\documentclass[12pt]{article}
\usepackage{latexsym}
\usepackage{amsmath,amssymb,amsthm}

\usepackage{amsmath,amsthm,latexsym,amssymb,pgf,makeidx,authblk,boxedminipage,tikz}
\usetikzlibrary{arrows,decorations.pathmorphing,backgrounds,positioning,fit,petri,calc}

\begin{document}

\newtheorem{defi}{Def}
\newtheorem{theo}{Theorem}
\newtheorem{coro}{Corollary}
\newtheorem{prop}{Proposition}
\newtheorem{rem}{Rem:}
\newtheorem{la}{Lemma}
\newtheorem{exa}{Example}
\newtheorem{conj}{Conjecture}
\newtheorem{problem}{Problem}

\begin{center}
\Large{Distance Sequences to bound the Harary Index and other Wiener-type Indices 
of a Graph} \\[2mm]
\large{Peter Dankelmann \\ University of Johannesburg}
\end{center}

\normalsize

\begin{abstract}
In this paper we obtain bounds on a very general class of distance-based topological indices 
of graphs, which includes the Wiener index,  defined as the sum of the distances
between all pairs of vertices of the graph, and most generalisations of the Wiener index,
including the Harary index and the hyper-Wiener index. 

Our results imply several new bounds on well-studied topological indices, among those sharp 
lower bounds on the Harary index and sharp upper bounds on the hyper-Wiener index for
(i) graphs of given order and size (which resolves a problem in the monograph 
[The Harary index of a graph, Xu, Das, Trinajsti\'{c}, Springer (2015)],   
(ii) for $\kappa$-connected graphs, where $\kappa$ is even,  
(iii) for maximal outerplanar graphs and for Apollonian networks (a subclass of maximal
planar graphs), and (iv) for trees in which all vertices have odd degree. 
\end{abstract}


\section{Introduction}

Many topological indices
used in chemical and pure graph theory are based on distances between vertices. 
The best-known among the distance-based topological indices is the Wiener index, defined as  
\[ W(G) = \sum_{ \{u,v\} \subseteq V(G)} d_G(u,v), \]
where $V(G)$ is the vertex set of $G$ and $d_G(u,v)$ denotes the usual shortest path distance 
between vertices $u$ and $v$ of $G$.
Two of the most important distance-based indices besides the Wiener index are the Harary index $H(G)$, 
defined as 
\[ H(G)  = \sum_{ \{u,v\} \subseteq V(G)} \frac{1}{d_G(u,v)}, \]
and the hyper-Wiener index $WW(G)$, defined as 
\[ WW(G) = \frac{1}{2} \sum_{ \{u,v\} \subseteq V(G)} \big(d_G(u,v)^2 + d_G(u,) \big). \]
Bounds on these indices in terms of graph properties or other graph invariants have 
been the subject of intense study in the literature. Since 
the extremal graphs for distance-based topological indices like the Wiener index, 
Harary index and hyper-Wiener index coincide in many cases, it is natural to explore 
ways to prove extremal results for these topological indices in a unified way. 
This approach lead to the study of generalised Wiener indices, (also called
$Q$-indices (see, for example, \cite{BruDosGraGut2011}), of the form
\[ W_f(G) =  \sum_{ \{u,v\} \subseteq V(G)} f(d_G(u,v)), \]
where $f$ is a real-valued function on the set of positive integers. For suitable choices
this definition includes many well-known distance-based indices, for example 
the Harary index, 
the variable Wiener index $W_{\lambda}$, where $\lambda \in \mathbb{R}-\{0\}$ defined by 
$W_{\lambda}(G) =  \sum_{ \{u,v\} \subseteq V(G)} (d_G(u,v))^{\lambda}$ see, for example, 
\cite{HriKnoSkr2019}), the hyper-Wiener index, 
the generalised hyper-Wiener index $WW_{\lambda}$, defined by $WW_{\lambda}(G) 
  =  \sum_{ \{u,v\} \subseteq V(G)} \frac{1}{2} (d_G(u,v)^{\lambda} + d_G(u,v)^{2\lambda})$, 
where $\lambda \in \mathbb{R}-\{0\}$ (see Tomescu, Jarmad and Kamil \cite{TomArsJam2015}),
and the Tratch, Stankevich, Zefirov index 
$\sum_{ \{u,v\} \subseteq V(G)} \frac{1}{3} 
     \big( d_G(u,v)+ \frac{1}{2} d_G(u,v)^2 + \frac{1}{6} d_G(u,v)^3 \big)$
(see \cite{TraStaZef1990, KleGut1999}).
While the multiplicative Wiener index, $\pi(G) = \prod_{ \{u,v\} \subseteq V(G)} d_G(u,v)$
cannot be expressed in the form $W_f(G)$ for any $f$, the logarithm
of $\pi(G)$ can. 

This unified approach has yielded numerous general bounds on $W_f$. (The following results
are stated for the case that $f$ is an increasing function, corresponding results hold 
if $f$ is decreasing.)
In \cite{LiuDas2018} the trees of given order that have the four largest and 
the four smallest values of $W_f$ were determined. 
Considering trees with a given degree sequence, it was shown in \cite{SchWagWan2012, WagWanZha2013} 
that the so-called greedy tree minimises $W_f$. 
In \cite{DanDos2021}, extremal trees minimising $W_f$ among all trees with a given 
eccentric sequence, i.e., the sequence of the eccentricities of the vertices, 
are obtained.  
Unicyclic graphs with extremal values of $W_f$ were considered in \cite{MarRod-arxiv}. 
Sufficient conditions for graphs in terms of $W_f$ to have vulnerability parameters 
(such as toughness, binding number, tenacity, integrity) at least a certain value 
were given in \cite{HuaLiu2024}. 
Graphs with given independence number or matching number that are extremal with respect to 
$W_f$ were considered in \cite{Cam2021}. 
In recent years, sufficient conditions in terms of $W_f$ that guarantee that graphs have
certain Hamiltonian properties have attracted attention, see 
\cite{AolLiuYuaYu2023, DenKuaWuHua2017, KuaHuaDen2016, ZhoWanLu2018, ZhoWanLu2019}.

Other approaches to a unified treatment of  topological indices on trees, that also includes some 
indices that were not distance-based, were taken in \cite{VukSed2018}, where a partial 
order on the set of trees of given order is defined, and results for indices that are
increasing or decreasing with respect to this partial order are obtained. This was further
applied in \cite{SonHuaWan2021}. 
In \cite{AndRazWag2021}, the authors presented very general conditions on topological indices 
under which the greedy tree is always extremal 
among trees with a given degree sequence. 

In this paper we take an entirely different approach to distance-based topological indices 
by utilising properties
of distance sequences. The {\em distance sequence} ${\cal D}(G)$ of a connected graph $G$ is the 
nondecreasing sequence of the distances between  all unordered pairs of distinct vertices of $G$. 
The distance sequence was first considered in \cite{CasDan2019}, where it was used to 
prove sharp bounds on the Wiener index of the strong product of graphs. 
We say that a topological index $I(G)$ is {\em distance-based} if it can be expressed in the form
$I(G) = g({\cal D}(G))$, where $g$ is a function defined on the set of all
finite sequences of positive integers. We say that $I$ is increasing (decreasing, nonincreasing,
nondecreasing) if the function $g$, when restricted to sequences of a fixed length, is 
increasing (decreasing, nonincreasing, nondecreasing) in every coordinate. This definition
includes as  special cases the generalised Wiener index $W_f$, defined as 
$\sum_{ \{u,v\} \subseteq V(G)} f(d_G(u,v))$, where $f$ is an increasing or 
decreasing real function (which in turn includes well-known topological indices such as 
the Harary index and the hyper-Wiener index), but also, for example, the diameter of a graph, 
defined as the largest of the distances between its vertices.  

Our approach yields several new bounds/extremal graphs for distance-based topological indices. 
We determine graphs that maximise (minimise) indreasing (decreasing) distance-based indices 
(i) among graphs of given order and size, 
(ii)among $\kappa$-connected graphs of given order, where $\kappa$ is even, 
(iii) for $k$-trees and maximal outerplanar graphs, and 
(iv) trees with all vertex degrees odd. 
In most cases, the bounds on the Harary index, the hyper-Wiener index or the multiplicative 
Wiener index implied by our results are new. The result under (i) above 
resolves an open problem from the monograph \cite{XuDasTri2015} on the Harary index.

This paper is organised as follows. In Section \ref{section:terminology}, we define the 
terminology and notation used in this paper. In Section \ref{section:distance sequences} 
we introduce the degree sequence of graphs and prove basic properties. 
In the following sections we obtain extremal graphs for increasing or decreasing 
distance-based topological indices for various graph classes:
graph of given order and size are considered in Section \ref{section:order-and-size},
graphs of given connectivity are considered in Section \ref{section:connectivity}, 
maximal $k$-degenerate graphs (which as special cases contain 
$k$-trees, maximal outerplanar graphs and Apollonian networks which are a subclass of maximal
planar graphs) are considered in Section \ref{section:outerplanar}. 
Section \ref{section:regular-trees} is on two classes of trees, regular trees and trees 
in which all vertices have odd degree.

\section{Terminology and notation}
\label{section:terminology}

If $G$ is a graph, then we denote its vertex set and edge set by $V(G)$ and $E(G)$. 
If $u,v$ are vertices of $G$, then we say that $u$ is a neighbour of $v$ if $uv$ is an
edge of $G$. The degree ${\rm deg}_G(v)$ of $v$ is the number of its neighbours. 

A graph is connected, if between any two of its vertices, $u$ and $v$ say, there is a 
$(u,v)$-path. The distance $d_G(u,v)$ between two vertices $u$ and $v$ in a connected 
graph is the minimum length of a $(u,v)$-path. If $u$ is a vertex of $G$, then
the eccentricity of $u$, denoted by ${\rm ecc}_G(u)$, is the distance from $u$ to a 
vertex farthest from $u$. 

If $S$ is a set of vertices of $G$, then $G-S$ denotes the graph obtained from $G$ by
deleting all vertices in $S$ and edges that are incident with a vertex in $S$.  
If $k\in \mathbb{N}$, we say that $G$ is $k$-connected if $G$ has more than $k$ vertices, and 
the graph $G-S$ is connected whenever $S$ is a subset of $V(G)$ with fewer than $k$ vertices. 

A graph $G$ is planar if it can be embedded in the plane such that no two edges intersect. 
A graph is outerplanar if it can be embedded in the plane such that no two edges 
intersect, and every vertex is on the boundary of the outer face. 
A graph $G$ is maximal planar (maximal outerplanar) if it is planar (outerplanar), 
but after adding any edge it no longer has this property. 

Let $G$ be a graph and $k\in \mathbb{N}$. By the $k$-th power $G^k$ of $G$ we mean the graph
on the same vertex set in which two vertices are adjacent if their distance is not more than $k$.

We denote the path, the cycle and the complete graph on $n$ vertices by $P_n$, $C_n$ and $K_n$, 
respectively.

\section{Distance sequences}
\label{section:distance sequences}

Recall that the {\em distance sequence} ${\cal D}(G)$ of a connected graph $G$ of order 
$n$ is the nondecreasing sequence of the distances between 
all unordered pairs of distinct vertices of $G$. The distance sequence was first
considered in \cite{CasDan2019}.

If $A=(a_1, a_2,\ldots, a_k)$ and $B=(b_1,b_2,\ldots,b_k)$ are nondecreasing sequences
of integers, then we write $A\leq B$ if $a_i\leq b_i$ for $i=1,2,\ldots,k$. 
We write $A<B$ if $A \leq B$ but $A\neq B$. 
The relation $\leq$ imposes a partial order on the distance sequences of connected
graphs of given order. 

If $I$ is an nondecreasing (nonincreasing) distance-based topological index, then clearly 
$I(G_1) \leq I(G_2)$ ($I(G_1) \geq I(G_2)$) 
whenever $G_1$ and $G_2$ are connected graphs of the same order 
with ${\cal D}(G_1) \leq {\cal D}(G_2)$. Hence we have the following proposition.

\begin{prop} \label{prop:application-of-distance-sequence}
Let ${\cal G}$ be a class of connected graphs, let $n \in \mathbb{N}$, and let 
${\cal G}_n$ be the set of all graphs of order $n$ in ${\cal G}$. \\
(a) Let $I$ be an nondecreasing distance-based topological index. If a graph 
$G_n \in {\cal G}_n$ satisfies ${\cal D}(G) \leq {\cal D}(G_n)$ for all $G \in {\cal G}_n$, 
then 
\[ I(G) \leq I(G_n) \quad \textrm{for all $G \in {\cal G}_n$}.   \] 
(b) Let $I$ be an increasing distance-based topological index. If a graph 
$G_n \in {\cal G}_n$ satisfies ${\cal D}(G) < {\cal D}(G_n)$ for all $G \in {\cal G}_n -\{G_n\}$, 
then 
\[ I(G) < I(G_n) \quad \textrm{for all $G \in {\cal G}_n - \{G_n\}$}.   \] 
Corresponding inequalities hold if $I$ is nonincreasing or decreasing.
\end{prop}

If $v$ is a vertex of a connected graph $G$, then ${\cal D}_G(v)$ denotes the
nondecreasing sequence of the distances between $v$ and all other vertices of $G$. 
The following lemma is a slightly more general version of the well-known fact 
that $W(G) \leq W(G-v) + d_G(v)$ if $v$ is not a cut-vertex of $G$, where 
$d_G(v)$ is the sum of the distances between $v$ and all other vertices of $G$. 

If $A$ and $B$ are sequences of integers, then we denote the sequence obtained from
the concatenation of $A$ and $B$ by arranging its element in nondecreasing order by $A \odot B$. 

\begin{la} \label{la:sequence-of-G-minus-v}
Let $G$ be a connected graph, and $v$ a vertex of $G$ which is not a cut-vertex. 
Then
\begin{equation} \label{eq:sequence-of-G-minus-v} 
{\cal D}(G) \leq {\cal D}(G-v) \odot {\cal D}_G(v),
\end{equation}
with equality if and only if $d_{G-v}(u,w) \leq 2$ for all $u,w \in N_G(v)$. 
\end{la}

\begin{proof}
Let ${\cal D}'$ be the nondecreasing sequence of the distances in $G$  
between all pairs of vertices in $V(G)-\{v\}$. Then
${\cal D}(G) = {\cal D}' \odot {\cal D}_G(v)$.  
Since $d_G(u,w) \leq d_{G-v}(u,w)$ for all $u,w \in V(G)-\{v\}$, we have 
${\cal D}' \leq {\cal D}(G-v)$. This proves \eqref{eq:sequence-of-G-minus-v}. 

Clearly, equality in \eqref{eq:sequence-of-G-minus-v} holds if an only if 
$d_G(u,w) = d_{G-v}(u,w)$ for all $u,w \in V(G)-\{v\}$, which in turn holds 
if and only if 
$d_G(u,w) = d_{G-v}(u,w)$ for all $u,w \in N_G(v)$. The latter holds if and only
$d_{G-v}(u,w) \leq 2$ for all $u,w \in N_G(v)$.
\end{proof}

\section{Graphs of given order and size}
\label{section:order-and-size}

In this section we consider graphs of given order and size. 
Sharp lower bounds on nondecreasing/increasing (or upper bounds on nonincreasing/decreasing) 
distance-based indices are easily derived from the fact that
a connected graph of order $n$ and size $m$ has $m$ pairs of vertices at distance $1$, 
and the remaining $\binom{n}{2}-m$ pairs have distance at least $2$. For example,
if $f$ is an increasing function then $W_f(G) \geq m f(1) + (\binom{n}{2}-m)f(2)$ for
every graph of order $n$ and size $m$, with equality if and only if ${\rm diam}(G) \leq 2$. 

Sharp upper bounds on nondecreasing/increasing (or lower bounds on noninccreasing/decreasing) 
distance-based indices are less straightforward to obtain. 
For the Wiener index, Solt\'{e}s \cite{Sol1991} proved that, 
among all connected graphs of given order $n$ and size $m$, the path-complete graph
$PK_{n,m}$ (defined below) is the unique graph maximising the Wiener index. 
No sharp upper bound on the hyper-Wiener index in terms of order and size is known, 
except for the special case $m=n$, i.e., for unicyclic graphs (see \cite{XinZhoQi2011}). 
We add that lower and upper bounds on the hyper-Wiener index in terms of order and 
size that also take into account eccentricities of vertices or diameter were given in 
in \cite{FenLiu2012} and \cite{AlhBagRah2017}. 
In their monograph on the Harary index \cite{XuDasTri2015}, Xu, Das and Trinajsti\'{c}  posed the 
problem of determining a sharp lower bound on the Harary index in terms of order and size. 
Theorem \ref{theo:order-size} below solves this problem not only for the Harary index, 
but for increasing distance-based topological indices. 

Following Solt\'{e}s \cite{Sol1991}, we define a path-complete graph as a graph obtained from the 
union of a path and a complete graph by joining one end of 
the path to one or more vertices of the complete graph. It is not hard to see that
for given $n,m \in \mathbb{N}$ with $n-1 \leq m < \binom{n}{2}$ there exists a unique
path-complete graph of order $n$ and size $m$. We denote this graph by $PK_{n,m}$.

Using Solt\'{e}s' approach, it was shown in \cite{CasDan2019}
that the distance sequence of $PK_{n,m}$ is maximal with respect to the partial order $\leq$ 
among the distance sequences of all connected graphs of order $n$ and size $m$.

\begin{la} \cite{CasDan2019} \label{la:path-complete-dominates-for-given-size}
Let $G$ be a connected graph of order $n$ and size $m$. Then
\[ {\cal D}(G) \leq {\cal D}(PK_{n,m}). \]
\end{la}

As a direct consequence of Lemma \ref{la:path-complete-dominates-for-given-size} and 
Proposition \ref{prop:application-of-distance-sequence} we obtain the following theorem. 

\begin{theo} \label{theo:order-size}
Let $G$ be a connected graph of order $n$ and size $m$. 
If $I$ is an nondecreasing distance-based topological index, then 
\[ I(G) \leq I(PK_{n,m}). \]
If $I$ is a nonincreasing distance-based topological index, then 
\[ I(G) \geq I(PK_{n,m}). \]
\end{theo}

Theorem \ref{theo:order-size} implies, for example, that among connected graphs of given order 
and size, the path-complete graph has minimum Harary index and maximum hyper-Wiener index as well
as maximum multiplicative Wiener index.

\section{Graphs of given connectivity}
\label{section:connectivity}

In this section we consider graphs of given connectivity $\kappa$. 
While sharp lower bounds for nondecreasing/increasing (or sharp upper bounds for nonincreasing/decreasing) distance-based topological indices of given order and
connectivity are known for many distance-based indices, (see \cite{BehJanTae2009} for the 
hyper-Wiener index, \cite{TomArsJam2015} for the general hyper-Wiener index, and 
\cite{FenZhoLiu2017} for upper bounds on the Harary index), 
a sharp upper bound is currently known only for Wiener index. For even $\kappa$, 
Favaron et al.\  \cite{FavKouMah1989} gave a sharp upper bound, while for
odd $\kappa$, the maximum Wiener index was determined asymptotically in
\cite{DanMukSwa2009}. So far, no upper bounds for hyper-Wiener index or lower bounds on Harary
index for graphs of connectivity greater than $1$ appear to be known. 
Below we give an upper bound on nondecreasing distance-based topological indices for 
graphs of order $n$ and connectivity $\kappa$, where $\kappa$ is even. 

In \cite{CasDan2019} it was shown that, for even $\kappa$, the 
distance sequence of the $\frac{\kappa}{2}$-th power of the cycle $C_n$, 
is maximal among the distance sequences of all $\kappa$-connected graphs of order $n$. 

\begin{la} \cite{CasDan2019} \label{la:cycle-power-dominates-for-given-connectivity}
Let $\kappa\in \mathbb{N}$ be even. If $G$ is a $\kappa$-connected graph of order $n$, then
\[ {\cal D}(G) \leq {\cal D}(C_n^{\kappa/2}). \] 
\end{la}

For the special case $\kappa=2$, i.e., for $2$-connected graphs, it is easy to see 
that equality in Lemma \ref{la:cycle-power-dominates-for-given-connectivity} 
implies that $G$ is $2$-regular, so for $\kappa=2$ equality holds if and only if $G=C_n$.

As a direct consequence of Lemma \ref{la:cycle-power-dominates-for-given-connectivity} and 
Proposition \ref{prop:application-of-distance-sequence} we obtain the following theorem.

\begin{theo} \label{theo:connectivity}
Let $G$ be a $\kappa$-connected graph of order $n$, where $\kappa$ is even. \\
(a) If $I$ is an nondecreasing distance-based topological index, then 
\[ I(G) \leq I(C_n^{\kappa/2}). \]
(b) If $\kappa=2$ and $I$ is a increasing distance-based topological index, then 
\[ I(G) \geq I(C_n^{\kappa/2}), \]
with equality if and only if $G$ is a cycle.  \\
(c) Corresponding statements hold if $G$ is nonincreasing or decreasing. 
\end{theo}

Theorem \ref{theo:connectivity} implies, for example that for even $\kappa$ the graph
$C_n^{\kappa/2}$ minimises the Harary index and maximises the hyper-Wiener index and the
multiplicative Wiener index among all $\kappa$-connected graphs of order $n$. 
For $\kappa=2$, the extremal graph is unique.

\section{Maximal (outer)planar graphs, $k$-trees, and $k$-degenerate graphs}
\label{section:outerplanar}

In this section we consider maximal outerplanar graphs and a subclass of maximal planar 
graphs, Apollonian networks, which we define below. We prove our results in a more general
setting, for maximal $k$-degenerate graphs, a superclass of $k$-trees.

Let $k\in \mathbb{N}$. A $k$-tree (see \cite{BeiPip1969}) is a graph defined as follows. 
The complete graph $K_{k+1}$ is a $k$-tree. If $G$ is a $k$-tree, then the graph obtained 
from $G$ by adding a new vertex and joining it to the vertices of a $k$-clique is also a $k$-tree. 
The $1$-trees are just the trees. 
A graph $G$ is $k$-degenerate if every induced
subgraph of $G$ contains a vertex of degree at most $k$ (see \cite{LicWhi1970}). A 
$k$-degenerate graph is maximal $k$-degenerate if after adding any edge it is no
longer $k$-degenerate. It was shown in \cite{Bic2012} that every $k$-tree is a maximal 
$k$-degenerate graph. 

It is easy to see that for $k \leq n-1$ the graph $P_{n}^{k}$, i.e., the $k$-th power of 
the path $P_n$, is a $k$-tree and thus $k$-degenerate. 
Bickle and Che \cite{BicChe2021} showed that the graph $P_n^k$ maximises the 
Wiener index among all maximal $k$-degenerate graphs, and thus among all $k$-trees, of order $n$. 
Modifying their proof slightly, we prove that this holds not only for the Wiener index,
but for all nondecreasing distance-based topological indices.

\begin{la}[\cite{Bic2012}] \label{la:k-degenerate-is-k-connected}
Every maximal $k$-degenerate graph is $k$-connected.
\end{la}

\begin{la}   \label{la:path-power-dominates-k-trees}
Let $k,n \in \mathbb{N}$ with $1 \leq k \leq n-1$. If $G$ is a maximal $k$-degenerate graph of 
order $n$, then
\[ {\cal D}(G) \leq {\cal D}(P_n^k). \] 
\end{la}

\begin{proof}
Let $k$ be fixed. We prove the statement by induction on $n$. 
If $n=k+1$, then $G$ is the complete graph $K_{k+1}$, so $G=P_n^{k}$ and the statement holds. 

Let $n > k+1$ and let $G$ be a maximal $k$-degenerate graph of order $n$. 
Since $G$ is $k$-degenerate, $G$ contains a vertex $v$ of degree at most $k$. Since $G$ is 
$k$-connected by Lemma \ref{la:k-degenerate-is-k-connected}, it follows that ${\rm deg}_G(v)=k$
and that $G-v$ is connected.  Also, $G-v$ is maximal $k$-degenerate. By our induction hypothesis we have  
\begin{equation} \label{eq:k-tree-1}
{\cal D}(G-v) \leq {\cal D}(P_{n-1}^k). 
\end{equation}
Since $G$ is $k$-connected, there are at least $k$ vertices at distance $i$ from $v$ for 
$i=1,2,\ldots,{\rm ecc}_G(v)-1$. Letting $s=\lfloor \frac{n-1}{k} \rfloor$,
we thus obtain
\begin{equation} \label{eq:k-tree-2} 
{\cal D}_G(v) \leq (1^{(k)}, 2^{(k)}, \ldots, (s)^{(k)},(s+1)^{(n-1-sk)}). 
\end{equation}
Let $w$ be an end vertex of the path $P_n$. It is easy to verify that 
${\cal D}(P_n^k,w) = (1^{(k)}, 2^{(k)}, \ldots, (s)^{(k)},(s+1)^{(n-1-sk)})$. Now
Lemma \ref{la:sequence-of-G-minus-v} in conjunction with
\eqref{eq:k-tree-1} and \eqref{eq:k-tree-2} yields 
\[ {\cal D}(G) \leq {\cal D}(G-v) \odot {\cal D}_G(v) 
               \leq {\cal D}(P_{n-1}^k) \odot {\cal D}_{P_{n}^k}(w) 
               = {\cal D}(P_{n}^k), \]
as desired.                
\end{proof}

As a direct consequence of Lemma \ref{la:k-degenerate-is-k-connected} and 
Proposition \ref{prop:application-of-distance-sequence} we obtain the following theorem.

\begin{theo} \label{theo:k-degenerate}
Let $k, n\in \mathbb{N}$ with $1 \leq k \leq n-1$ and let $G$ be a maximal $k$-degenerate graph of order $n$. \\
If $I$ is a nondecreasing distance-based topological index, then 
\[ I(G) \leq I(P_n^{k}). \]
If $I$ is a nonincreasing distance-based topological index, then 
\[ I(G) \geq I(P_n^{k}). \]
\end{theo}

Since every $k$-tree is maximal $k$-degenerate, and $P_n^k$ is a $k$-tree, Theorem \ref{theo:k-degenerate} 
yields that $P_n^k$ maximises every nondecreasing (minimises every nondecreasing) distance-based topological
index for $k$-trees of order $n$.
 
It is well-known that every maximal outerplanar graph is a $2$-tree. 
Among the maximal planar graphs, the planar $3$-trees, also called Apollonian networks,
are of interest here. They can be thought of as planar graphs that are obtained from
a triangle by successively inserting a new vertex into a face of length three and joining
the new vertex to the three vertices on this face by edges. 
It is easy to verify that the $2$-tree $P_n^2$ is maximal outerplanar, and the 
$3$-tree $P_n^3$ is an Apollonian network for $n\geq 3$. Hence we have the following 
corollary. 

\begin{coro}  \label{coro:topololgical-indices-of-maximal-outerplanar-graphs}
(a) Let $G$ be a maximal outerplanar graph of order $n \geq 3$. 
If $I$ is a nondecreasing (nonincreasing) distance-based topological index, then 
\[ I(G) \leq I(P_n^{2}) \quad ( I(G) \geq I(P_n^{2})). \]
(a) Let $G$ be an Apollonian network of order $n \geq 3$. 
If $I$ is a nondecreasing (nonincreasing) distance-based topological index, then 
\[ I(G) \leq I(P_n^{3}) \quad ( I(G) \geq I(P_n^{3})). \]
\end{coro}

Corollary \ref{coro:topololgical-indices-of-maximal-outerplanar-graphs} implies, for example, 
that among all maximal outerplanar graphs of order $n$, the graph $P_n^2$ has minimum Harary index 
and maximum hyper-Wiener index as well as maximum multiplicative Wiener index.  Also, 
among all Apollonian networks of order $n$, the graph $P_n^3$ has minimum Harary index 
and maximum hyper-Wiener index as well as maximum multiplicative Wiener index. 

Corollary \ref{coro:topololgical-indices-of-maximal-outerplanar-graphs}(b) also contains as a special case
a sharp upper bound on the Wiener index of Apollonian networks, which was proved by 
Bickle and Che \cite{BicChe2021}.

We do not know if the conclusion of 
Corollary \ref{coro:topololgical-indices-of-maximal-outerplanar-graphs}(b) holds for all 
maximal planar graphs. If this is the case, this would yield, among others, a sharp bound
on Harary index and hyper-Wiener index for maximal planar graphs.

\section{Trees with all degrees odd}
\label{section:regular-trees}

In this section we consider trees in which every vertex has odd degree. We refer to
such trees as odd trees. 
The minimum and maximum Wiener index of odd trees was determined by \cite{Lin2013}. For 
further results on the Wiener index of odd trees see \cite{ForGutLin2013, For2013}. 
For other distance-based topological indices, the maximum or minimum value for odd trees
appears not to have been investigated. 

It follows from the handshake lemma that odd trees have even order. 
Let $n \in \mathbb{N}$ be even, $n\geq 4$. Let $S_n$ be the star of order $n$.
If $n\geq 6$ define 
the tree $T_n$ as the tree obtained from a path $P_{n/2+1}$ by attaching a leaf to each internal vertex 
of the path, and if $n=4$ define $T_n=S_n$.  Clearly, $T_n$ is an odd tree of order $n$.

\begin{la}   \label{la:odd-trees}
Let $T$ be an odd tree of order $n$, $n \geq 4$. Then
\[  {\cal D}(T) \leq {\cal D}(T_{n}). \]
\end{la}

\begin{proof}
%
We prove the bound on ${\cal D}(T)$ by induction on $n$. 
If $n=4$, then $T$ equals $S_4$ and the lemma holds. 
Now assume that or $n\geq 6$, and assume further that the lemma holds for all odd trees 
of order less than $n$.
Let $P$ be a longest path in $T$. Let $v_1$ be its 
end vertex and $w$ the unique neighbour of $v$. Since ${\rm deg}_T(w)$ is odd,
vertex $w$ has at least one other leaf neighbour $v_2$. 
Consider the tree $T-\{v_1,v_2\}$.  Observe that $T-\{v_1,v_2\}$ is an odd tree. 
We express ${\cal D}(T)$ in terms of ${\cal D}(T-\{v_1, v_2\}$ and ${\cal D}_{T-v_2}(v_1)$. 
Clearly, $v_1$ and $v_2$ have the same distance sequence in $T$, and 
${\cal D}_T(v_1) = {\cal D}_{T}(v_2) = {\cal D}_{T-v_2}(v_1) \odot (2)$. Hence
\begin{equation} \label{eq:odd-tree-1} 
{\cal D}(T) = {\cal D}_{T-\{v_1,v_2\}} \odot {\cal D}_{T-v_2}(v_1) 
       \odot {\cal D}_{T-v_2}(v_1) \odot (2). 
\end{equation}       
We bound the terms on the right hand side of \eqref{eq:odd-tree-1} separately. 
Applying our inductive hypothesis to the odd tree $T-\{v_1,v_2\}$ yields 
\begin{equation} \label{eq:odd-tree-2}
{\cal D}(T-\{v_1, v_2\}) \leq {\cal D}(T_{n-2}). 
\end{equation}
Now consider ${\cal D}_{T-v_2}(v_1)$. Apart from the neighbour $w$ of $v$ in $T'$, every
internal vertex of $T'$ has degree at least $3$. This implies that there are at least two vertices
at distance $j$ from $v_1$ for every $j \in \{3,4,\ldots,{\rm ecc}_{T-v_2}(v_1)\}$.
This in turn implies that
\begin{equation} \label{eq:odd-tree-3}
{\cal D}_{T'}(v_1) \leq (1, 2, 3^{(2)}, 4^{(2)},\ldots,(\frac{n}{2})^{(2)}).
\end{equation} 
Substituting  \eqref{eq:odd-tree-2} and \eqref{eq:odd-tree-3} into \eqref{eq:odd-tree-1}, 
we obtain 
\begin{eqnarray*}
{\cal D}(T) & \leq & {\cal D}(T_{n-2}) \oplus  (1, 2, 3^{(2)}, 4^{(2)},\ldots,(\frac{n}{2})^{(2)}) \\
  & & \oplus (1, 2, 3^{(2)}, 4^{(2)},\ldots,(\frac{n}{2})^{(2)}) \odot (2) \\
  & =  &  {\cal D}(T_{n-2}) \oplus  (1^{(2)}, 2^{(3)}, 3^{(4)}, 4^{(4)},\ldots,(\frac{n}{2})^{(4)}).
\end{eqnarray*} 
It is now easy to verify that 
${\cal D}(T_n) = {\cal D}(T_{n-2}) \oplus 
(1^{(2)}, 2^{(3)}, 3^{(4)}, 4^{(4)},\ldots,(\frac{n}{2})^{(4)})$ for
$n\geq 6$. Hence we obtain 
that  ${\cal D}(T) \leq {\cal D}(T_{n})$, as desired.
\end{proof}

As a direct consequence of Lemma \ref{la:odd-trees} and 
Proposition \ref{prop:application-of-distance-sequence} we obtain the following theorem.

\begin{theo} \label{theo:odd-trees}
Let $T$ be an odd tree of order $n$, where $n\geq 4$. \\
If $I$ is a nondecreasing distance-based topological index, then 
\[ I(G) \leq I(T_n). \]
If $I$ is a nonincreasing distance-based topological index, then 
\[ I(G) \geq I(T_n). \]
\end{theo}

Theorem \ref{theo:odd-trees} implies, for example, that $T_n$ minimises the Harary index
and maximises the hyper-Wiener index and the multiplicative Wiener index among odd trees
of order $n$.

\section{Conclusion}

The approach taken in this paper, to bound distance-based topological indices by considering
the distance sequence appears very fruitful. It would be useful to find further graph 
classes in which a similar approach works.

\end{document}